\documentclass{amsart}
\usepackage{amssymb,amsmath,amsfonts,amscd,amsthm,pb-diagram}
\usepackage{amsbsy,bm}
\usepackage[usenames,dvipsnames]{color}
\usepackage[normalem]{ulem}

\newtheorem{theorem}{Theorem}[section]

\newtheorem{lemma}{Lemma}[section]

\newtheorem{corollary}{Corollary}[section]
\newtheorem{remark}{Remark}[section]

\newtheorem*{theorem A}{Theorem A}
\newtheorem*{main theorem}{Main Theorem}

\begin{document}
\title[Characterizations of the Whitney and contact Whitney spheres]
{New characterizations of the Whitney spheres and the contact Whitney spheres}

\author{Zejun Hu and Cheng Xing}%----------Author 1,2
\address{%
School of Mathematics and Statistics, Zhengzhou University,
Zhengzhou 450001, People's Republic of China.}
\email{huzj@zzu.edu.cn; xingchengchn@yeah.net}


\thanks{This project was supported by NSF of China, Grant Number 11771404.}

\date{}


%%%%%%%%%%%%%%%%%%%%%%%%%%%%%%%%%%%%%%%%%%%%%%%%%%%%%%%%%%%%%%
\begin{abstract}
In this paper, based on the classical K. Yano's formula, we first establish an optimal
integral inequality for compact Lagrangian submanifolds in the complex space forms,
which involves the Ricci curvature in the direction $J\vec{H}$ and the norm of the
covariant differentiation of the second fundamental form $h$, where $J$ is the almost
complex structure and $\vec{H}$ is the mean curvature vector field. Second and analogously,
for compact Legendrian submanifolds in the Sasakian space forms with Sasakian structure
$(\varphi,\xi,\eta,g)$, we also establish an optimal integral inequality involving
the Ricci curvature in the direction $\varphi\vec{H}$ and the norm of the modified 
covariant differentiation of the second fundamental form. The integral inequality is 
optimal in the sense that all submanifolds attaining the equality are completely classified.
As direct consequences, we obtain new and global characterizations for the Whitney
spheres in complex space forms as well as the contact Whitney spheres in Sasakian
space forms. Finally, we show that, just as the Whitney spheres in complex space forms, 
the contact Whitney spheres in Sasakian space forms are locally conformally flat manifolds 
with sectional curvatures non-constant.
\end{abstract}

\maketitle
%%%%%%%%%%%%%%%%%%%%%%%%%%%%%%%%%%%%%%%%%%%%%%%%%%%%%%%%%%%%%%
\numberwithin{equation}{section}

%%%%%%%%%%%%%%%%%%%%%%%%%%%%%%%%%%%%%%%%%%%%%%%%%%%%%%%%%%%%%%
\section{Introduction}\label{sect:1}
%%%%%%%%%%%%%%%%%%%%%%%%%%%%%%%%%%%%%%%%%%%%%%%%%%%%%%%%%%%%%%

In this paper, we consider compact Lagrangian submanifolds in the $n$-dimensional
complex space form $N^n(4c)$ of constant holomorphic sectional curvature $4c$,
$c\in\{0,1,-1\}$; and analogously, we also consider compact Legendrian submanifolds
in the $(2n+1)$-dimensional Sasakian space form $\tilde{N}^{2n+1}(\tilde{c})$ with
constant $\varphi$-sectional curvature $\tilde{c}$. As our main achievements, we
shall establish an integral inequality for either class of such submanifolds. Then,
as direct consequences, we can get new and integral characterizations for the
Whitney spheres in the complex space forms and also the contact Whitney spheres in
the Sasakian space forms.

Recall that the complex space form $N^n(4c)$ with almost complex structure $J$
and Riemannian metric $g$ is the complex Euclidean space $\mathbb{C}^n$ for $c=0$,
the complex projective space $\mathbb{C}P^n(4)$ for $c=1$, and the complex hyperbolic
space $\mathbb{C}H^n(-4)$ for $c=-1$. Let $M^n\hookrightarrow N^n(4c)$ be a {\it Lagrangian}
immersion of an $n$-dimensional differentiable manifold $M^n$ ($n\ge2$), i.e.,
$J$ carries each tangent space of $M^n$ into its corresponding normal space. In order
to state our first main result, we shall recall the notion of {\it Whitney spheres} in
each complex space form.

\vskip3mm\noindent
{\it Example 1.1}: {\bf Whitney spheres} in $\mathbb{C}^n$ (cf. \cite{BCM,C1,CU,LV,RU,SS}).

As the most classical notion of {\it Whitney spheres}, these are usually defined as a family
of Lagrangian immersions from the unit sphere $\mathbb{S}^n$, centered at the origin $O$
of $\mathbb{R}^{n+1}$, into the complex Euclidean space $\mathbb{C}^n\cong\mathbb{R}^{2n}$,
given by $\Psi_{r,B}: \mathbb{S}^n\rightarrow\mathbb{C}^n$ with
\begin{equation}\label{eqn:1.1}
\Psi_{r,B}(u_1,\ldots,u_{n+1})
=\tfrac{r}{1+u^2_{n+1}}(u_1,u_1u_{n+1},\ldots,u_n,u_nu_{n+1})+B,
\end{equation}
where $r$ is a positive number and $B$ is a vector of $\mathbb{C}^n$. The number $r$ and
the vector $B$ are called the radius and the center of the Whitney spheres, respectively.
Up to translation and scaling of $\mathbb{C}^n$, all the Whitney spheres are congruent
with the standard one corresponding to $r=1$ and $B=O$.
According to Gromov \cite{G}, the sphere cannot be embedded into $\mathbb{C}^n$ as a
Lagrangian submanifold. This fact implies that the Whitney spheres in \eqref{eqn:1.1}
have the best possible behavior, because it is embedded except at the poles of $\mathbb{S}^n$
where it has a double points. Indeed, in a certain sense, the Whitney spheres in $\mathbb{C}^n$
play the role of umbilical hypersurfaces of the Euclidean space $\mathbb{R}^{n+1}$ inside
the family of Lagrangian submanifolds and have been characterized in several ways as done
for the Euclidean spheres (cf. \cite{RU}).

\vskip3mm\noindent
{\it Example 1.2}: {\bf Whitney spheres} in $\mathbb{C}P^n(4)$ (cf. \cite{CMU,CU1,CV,LV}).

In this case, the {\it Whitney spheres} are a one-parameter family of Lagrangian sphere
immersions into $\mathbb{C}P^n(4)$, given by $\Psi_\theta:\mathbb{S}^n\rightarrow\mathbb{C}P^n(4)$
for $\theta>0$ with
\begin{equation}\label{eqn:1.2}
\Psi_\theta(u_1,\dots,u_{n+1})=\Pi\left(\tfrac{(u_1,\ldots,u_n)}{\cosh
\theta+i\sinh\theta u_{n+1}};\tfrac{\sinh\theta\cosh\theta(1+u^2_{n+1})
+iu_{n+1}}{\cosh^2\theta+\sinh^2\theta u^2_{n+1}}\right),
\end{equation}
where $\Pi:\mathbb{S}^{2n+1}\rightarrow\mathbb{C}P^n(4)$ is the Hopf projection.
We notice that $\Psi_\theta$ are embeddings except at the poles of $\mathbb{S}^n$
where it has a double points, and that $\Psi_0$ is the totally geodesic Lagrangian
immersion of $\mathbb{S}^n$ into $\mathbb{C}P^n(4)$.

\vskip3mm\noindent
{\it Example 1.3}: {\bf Whitney spheres} in $\mathbb{C}H^n(-4)$ (cf. \cite{CMU,CU1,CV,LV}).

Let $(\cdot,\cdot)$ denote the hermitian form of $\mathbb{C}^{n+1}$, i.e.,
$(z,w)=\sum\limits_{i=1}^nz_i\bar{w}_i-z_{n+1}\bar{w}_{n+1}$ for $z,w\in\mathbb{C}^{n+1}$,
and $\mathbb{H}^{2n+1}_1(-1)=\{z\in\mathbb{C}^{n+1}:\,(z,z)=-1\}$ be the Anti-de Sitter space
of constant sectional curvature $-1$.
Then, the {\it Whitney spheres} in $\mathbb{C}H^n(-4)$ are a one-parameter family of Lagrangian sphere
immersions into $\mathbb{C}H^n(-4)$, given by $\Phi_\theta:\mathbb{S}^n\rightarrow\mathbb{C}H^n(-4)$
for $\theta>0$ with
\begin{equation}\label{eqn:1.3}
\Phi_\theta(u_1,\ldots,u_{n+1})=\Pi\left(\tfrac{(u_1,\ldots,u_n)}{\sinh\theta+i\cosh\theta u_{n+1}};
\tfrac{\sinh\theta\cosh\theta(1+u^2_{n+1})-iu_{n+1}}{\sinh^2\theta+\cosh^2\theta u^2_{n+1}}\right),
\end{equation}
where $\Pi:\mathbb{H}^{2n+1}_1(-1)\rightarrow\mathbb{C}H^n(-4)$ is the Hopf projection. We also
notice that $\Phi_\theta$ are embeddings except in double points.

\vskip2mm
The remarkable properties of the Whitney spheres are summarized as follows:

\begin{theorem}[cf. \cite{BCM,C2,CMU,CU,CU1,CV,LV,RU}]\label{thm:1.1}
Let $x:M^n\to N^n(4c)$ be an $n$-dimensional compact Lagrangian submanifold that is
neither totally geodesic nor of parallel mean curvature vector field. Then,
$x(M^n)$ is the Whitney sphere in $N^n(4c)$ if and only if one of the following
pointwise relations holds:

{\rm(1)} The squared mean curvature $|\vec{H}|^2$ and the scalar curvature $R$ of $M^n$ satisfy
the relation $|\vec{H}|^2=\tfrac{n+2}{n^2(n-1)}R-\tfrac{n+2}{n}c$;

{\rm(2)} The second fundamental form $h$ and the mean curvature vector field $\vec{H}$ of $M^n$
satisfy $h(X,Y)=\tfrac{n}{n+2}\big[g(X,Y)\vec{H}+g(JX,\vec{H})JY+g(JY,\vec{H})JX\big]$ for $X,Y\in TM^n$;

{\rm(3)} The second fundamental form $h$ and the mean curvature vector field $\vec{H}$ of $M^n$
satisfy $\|\bar\nabla h\|^2=\tfrac{3n^2}{n+2}\|\nabla^\perp\vec{H}\|^2$. Here, $\bar\nabla h$ denotes
the covariant differentiation of $h$ with respect to the van der Waerden-Bortolotti connection of
$x:M^n\rightarrow N^n(4c)$.
\end{theorem}

Moreover, Castro-Montealegre-Urbano \cite{CMU} and Ros-Urbano \cite{RU} further proved
that the Whitney spheres in $N^n(4c)$ can be characterized by some other relations about
the global geometric and topological invariants.

\vskip2mm
As the first main result of this paper, we have obtained an optimal integral inequality
that involves the Ricci curvature ${\rm Ric}\,(J\vec{H},J\vec{H})$ in the direction $J\vec{H}$
and the norm of the covariant differentiation $\bar{\nabla} h$ of the second fundamental
form:

\begin{theorem}\label{thm:1.2}
Let $x:M^n\rightarrow N^n(4c)\ (n\ge2)$ be an $n$-dimensional compact Lagrangian submanifold.
Then, it holds that
\begin{equation}\label{eqn:1.4}
\int_{M^n}{\rm Ric}\,(J\vec{H},J\vec{H})~dV_{M^n}\leq
\tfrac{(n-1)(n+2)}{3n^2}\int_{M^n}\|\bar{\nabla} h\|^2~dV_{M^n},
\end{equation}
where $\|\cdot\|$ and $dV_{M^n}$ denote the tensorial norm and the volume element of $M^n$
with respect to the induced metric, respectively.

Moreover, the equality in \eqref{eqn:1.4} holds if and only if either $x(M^n)$
is of parallel second fundamental form, or it is one of the Whitney spheres
in $N^n(4c)$.
\end{theorem}

\begin{remark}\label{rm:1.1}
The classification of Lagrangian submanifolds with parallel second fundamental form
in $N^n(4c)$ has been fulfilled for each $c$, see \cite{DLVW,HY} for details.
\end{remark}

From Theorem \ref{thm:1.2}, we get a new and global geometric characterization of the
Whitney spheres in $N^n(4c)$:

\begin{corollary}\label{cor:1.1}
Let $x:M^n\rightarrow N^n(c)\ (n\ge2)$ be an $n$-dimensional compact Lagrangian
submanifold with non-parallel mean curvature vector field. Then,
\begin{equation}\label{eqn:1.5}
\int_{M^n}{\rm Ric}\,(J\vec{H},J\vec{H})~dV_{M^n}=
\tfrac{(n-1)(n+2)}{3n^2}\int_{M^n}\|\bar{\nabla} h\|^2~dV_{M^n}
\end{equation}
holds if and only if $x(M^n)$ is a Whitney sphere in $N^n(c)$.
\end{corollary}

\vskip2mm
Next, before stating our second main result, we shall first review the standard models of
the Sasakian space form $\tilde{N}^{2n+1}(\tilde{c})$ with Sasakian structure
$(\varphi,\xi,\eta,g)$ possessing constant $\varphi$-sectional curvature $\tilde{c}$,
then for each value $\tilde{c}$ we introduce the canonical {\it Legendrian} (i.e., {\it the $n$-dimensional
$C$-totally real}, or equivalenly, {\it integral}) submanifolds:
The contact Whitney spheres in $\tilde{N}^{2n+1}(\tilde{c})$.

\vskip3mm\noindent
{\it Example 1.4}: {\bf Contact Whitney spheres} in $\tilde{N}^{2n+1}(-3)=(\mathbb R^{2n+1},\varphi,\xi,\eta,g)$.

Here, for the Cartesian coordinates $(x_1,\ldots,x_n,y_1,\ldots,y_n,z)$ of $\mathbb{R}^{2n+1}$,
\begin{equation*}
\left\{
\begin{aligned}
&\xi=2\tfrac{\partial}{\partial z}, \ \ \eta=\tfrac12\Big(dz-\sum_{i=1}^ny_idx_i\Big), \ \
g=\eta\otimes\eta+\tfrac14\sum_{i=1}^n(dx_i\otimes dx_i+dy_i\otimes dy_i),\\[-1mm]
&\varphi\Big(\sum_{i=1}^n\big(X_i\tfrac{\partial}{\partial
x_i}+Y_i\tfrac{\partial}{\partial y_i}\big)+Z\tfrac{\partial}{\partial z}\Big)
=\sum_{i=1}^n\big(Y_i\tfrac{\partial}{\partial x_i}-X_i\tfrac{\partial}{\partial
y_i}\big)+\sum_{i=1}^nY_iy_i\tfrac{\partial}{\partial z},
\end{aligned}
\right.
\end{equation*}
define the standard Sasakian structure $(\varphi,\xi,\eta,g)$ on $\mathbb{R}^{2n+1}$.

As were introduced by Blair and Carriazo in \cite{BC}, the contact Whitney spheres in
$\tilde{N}^{2n+1}(-3)$ were the Legendrian imbeddings
$\tilde\Psi_{B,a,r}:\mathbb{S}^n\rightarrow\mathbb{R}^{2n+1}$ defined by
\begin{equation}\label{eqn:1.6}
\tilde\Psi_{B,a,r}(u_0,u_1,\ldots,u_n)=\tfrac{r}{1+u^2_0}\big(u_0u_1,\ldots,u_0u_n,u_1,\dots,
\tfrac{ru_0}{1+u^2_0}+a(1+u^2_0)\big)+B,
\end{equation}
where $r$ is a positive number, $a$ is a real constant and $B$ is a vector of
$\mathbb{R}^{2n+1}$.

\vskip3mm\noindent
{\it Example 1.5}: {\bf Contact Whitney spheres} in $\tilde{N}^{2n+1}(\tilde{c})
=(\mathbb S^{2n+1},\varphi,\xi,\eta,g)$ with $\tilde{c}>-3$.
Note that the unit sphere $\mathbb S^{2n+1}$, as a real hypersurface of
the complex Euclidean space $\mathbb C^{n+1}$, has a natural Sasakian
structure $(\bar\varphi,\bar\xi,\bar\eta,\bar g)$:
$\bar g$ is the induced metric;
$\bar\xi=JN$, where $J$ is the natural complex
structure of $\mathbb C^{n+1}$ and $N$ is the unit normal vector
field of the inclusion $\mathbb S^{2n+1}\hookrightarrow\mathbb
C^{n+1}$;
$\bar\eta(X)=\bar{g}(X,\bar\xi)$ and $\bar\varphi(X)=JX-\langle
JX,N\rangle N$ for any tangent vector field $X$ of $\mathbb
S^{2n+1}$, where $\langle\cdot,\cdot\rangle$ denotes the standard
Hermitian metric on $\mathbb C^{n+1}$.
Then, the standard Sasakian structure $(\varphi,\xi,\eta,g)$ on
$\mathbb S^{2n+1}$ is given by applying a $D_a$-homothetic
deformation as follows:
\begin{equation*}
\eta=a\bar\eta,\ \ \xi=\tfrac1a\bar\xi,\ \ \varphi=\bar\varphi,\ \
g=a\bar g+a(a-1)\bar\eta\otimes\bar\eta,
\end{equation*}
where $a$ is a positive real number and $\tilde{c}=\tfrac4a-3$.

Then, as were introduced in \cite{HY}, the contact Whitney spheres in $\tilde{N}^{2n+1}(\tilde{c})$ for $\tilde{c}>-3$
are a family of Legendrian immersions $\tilde\Psi_\theta:\mathbb{S}^n\rightarrow\mathbb{S}^{2n+1}$ for $\theta>0$,
that are explicitly given by
\begin{equation}\label{eqn:1.7}
\tilde\Psi_{\theta}(u_1,u_2,\ldots,u_{n+1})=\Big(\tfrac{(u_1,\ldots,u_n)}{\cosh\theta+i\sinh
\theta u_{n+1}}; \tfrac{\sinh
\theta\cosh\theta(1+u^2_{n+1})+iu_{n+1}}{\cosh^2\theta+\sinh^2\theta
u^2_{n+1}}\Big).
\end{equation}

\vskip3mm
{\it Example 1.6}: {\bf Contact Whitney spheres} in $\tilde{N}^{2n+1}(\tilde{c})
=(\mathbb{B}^n\times\mathbb{R},\varphi,\xi,\eta, g)$ with $\tilde{c}<-3$.
Here, $\mathbb{B}^n=\{(z_1, \ldots, z_n)\in \mathbb{C}^n;\ \|z\|^2=\sum\limits_{i=1}^n|z_i|^2< 1\}$
equipped with the usual complex structure and the canonical Bergman metric
$$
\tilde{g}=4\Big\{\tfrac{1}{1-\|z\|^2}\sum_{i=1}^ndz_id\bar{z}_i
+\tfrac{1}{(1-\|z\|^2)^{2}}\sum_{i,j=1}^nz_i\bar{z}_jdz_jd\bar{z}_i\Big\}
$$
is a K\"ahler manifold with constant holomorphic sectional curvature $-1$.
Let $t$ be the coordinate of $\mathbb{R}$ and
$\omega=\frac{2\sqrt{-1}}{1-\|z\|^2}\sum\limits_{j=1}^n(\bar{z}_jdz_j-z_jd\bar{z}_j)$.
Then, $\mathbb{B}^n\times\mathbb{R}$ has a Sasakian structure $\{\bar{\varphi},
\bar{\xi},\bar{\eta},\bar{g}\}$ with constant $\bar{\varphi}$-sectional curvature $-4$,
defined as follows:
\begin{equation}\nonumber
\left\{
\begin{aligned}
&\bar{\eta}=\omega+dt, \ \ \bar{\xi}=\tfrac{\partial}{\partial t}, \ \ \bar{g}
=\tilde{g}+\bar{\eta}\otimes \bar{\eta},\\
&\bar{\varphi}\,\Big(\sum_{i=1}^{n}(a_i\tfrac{\partial}{\partial z_i}
+b_i\tfrac{\partial}{\partial \bar{z}_i})+e\tfrac{\partial}{\partial t}\Big)\\
&=\sqrt{-1}\sum_{i=1}^n(b_i\tfrac{\partial}{\partial z_i}
-a_i\tfrac{\partial}{\partial \bar{z}_i})+\tfrac{2}{1-\|z\|^2}\sum_{i=1}^n
(b_i\bar{z}_i+a_iz_i)\tfrac{\partial}{\partial t}.
\end{aligned}\right.
\end{equation}
Then, $\tilde{N}^{2n+1}(\tilde{c})=(\mathbb{B}^{n}\times\mathbb{R}, \varphi, \xi, \eta, g)$
is given by the $D_a$-homothetic deformation
$$
\eta=a\bar{\eta}, \ \ \xi=\tfrac{1}{a}\bar{\xi}, \ \ g=a\bar{g}+a(a-1)\bar{\eta}\otimes\bar{\eta},
$$
and $\tilde{c}=-\frac{1}{a}-3$, where $a$ is a positive number.

As were introduced in \cite{HY}, the contact Whitney spheres in
$\tilde{N}^{2n+1}(\tilde{c})$ for $\tilde{c}<-3$ are a one-parameter
family of Legendrian immersions
$\tilde\Phi_{\theta}:\mathbb{S}^n\rightarrow\mathbb{B}^n\times\mathbb{R}$ for $\theta>0$,
that are explicitly given by
\begin{equation}\label{eqn:1.8}
\pi(\tilde\Phi_{\theta}(u_1,u_2,\ldots,u_{n+1}))=\Pi\Big(\tfrac{(u_1,\ldots,u_n)}
{\cosh\theta+i\sinh\theta {u}_{n+1}}; \tfrac{\sinh\theta\cosh\theta(1+u^2_{n+1})
-iu_{n+1}}{\cosh^2\theta+\sinh^2\theta u^2_{n+1}}\Big),
\end{equation}
where $\pi:\tilde{N}^{2n+1}(\tilde{c})\rightarrow N^n(c)$ with $c=\tilde{c}+3$ is the canonical projection
and $\Pi:\mathbb{H}^{2n+1}_1(-1)\rightarrow\mathbb{C}H^n(-4)$ is the Hopf projection.

\vskip2mm
According to Proposition 2 of Blair-Carriazo \cite{BC} and Theorem 4.2 of Hu-Yin \cite{HY},
for each contact Whitney sphere $M^n$ in any Sasakian space form $\tilde{N}^{2n+1}(\tilde{c})$,
the second fundamental form $h$ and the mean curvature vector field $\vec{H}$ satisfy the relation
\begin{equation}\label{eqn:1.9}
h(X,Y)=\tfrac{n}{n+2}\big[g(X,Y)\vec{H}+g(\varphi X,\vec{H})\varphi Y+g(\varphi Y,\vec{H})\varphi X\big]
\end{equation}
for any tangent vector fields $X,Y\in TM^n$. Without introducing the notion of contact Whitney
spheres as canonical examples, Piti\c{s} \cite{P} proved that results for Lagrangian submanifolds
of the complex space forms in \cite{CMU,RU} hold analogously for those Legendrian submanifolds of
the Sasakian space forms which satisfy \eqref{eqn:1.9}. Moreover, an analogue of the result for
Whitney spheres in complex space forms by Li-Vrancken \cite{LV} was established for contact Whitney
spheres in Sasakian space forms by Hu-Yin \cite{HY}. It follows that a corresponding version of
Theorem \ref{thm:1.1} for contact Whitney spheres in Sasakian space forms is already known.

Next, along a similar spirit as above, we get the second main result of this paper. Actually,
we can show that an optimal integral inequality, as in Theorem \ref{thm:1.2}, that involves the
Ricci curvature ${\rm Ric}\,(\varphi\vec{H},\varphi\vec{H})$ in the direction $\varphi\vec{H}$ and the
norm of the modified covariant differentiation $\bar\nabla^\xi h$ of the second fundamental form
holds also for compact Legendrian submanifolds in the Sasakian space forms:

\begin{theorem}\label{thm:1.3}
Let $x:M^n\rightarrow\tilde{N}^{2n+1}(\tilde{c})$ be an $n$-dimensional compact Legendrian
submanifold. Then, it holds that
\begin{equation}\label{eqn:1.10}
\int_{M^n}{\rm Ric}\,(\varphi\vec{H},\varphi\vec{H})~dV_{M^n}\leq
\tfrac{(n-1)(n+2)}{3n^2}\int_{M^n}\|\bar\nabla^\xi h\|^2~dV_{M^n},
\end{equation}
where, $\bar\nabla^\xi h$ denotes the projection of $\bar\nabla h$ onto $TM^n\oplus\varphi(TM^n)$,
whereas $\bar\nabla h$ denotes the covariant differentiation of $h$ with respect to the van
der Waerden-Bortolotti connection of $M^n\hookrightarrow\tilde{N}^{2n+1}(\tilde{c})$,
$\|\cdot\|$ and $dV_{M^n}$ denote the tensorial norm and the volume element of $M^n$
with respect to the induced metric, respectively.

Moreover, the equality in \eqref{eqn:1.10} holds if and only if either $\bar\nabla^\xi h=0$
(i.e. $x(M^n)$ is of $C$-parallel second fundamental form), or $x(M^n)$ is
one of the contact Whitney spheres in $\tilde{N}^{2n+1}(\tilde{c})$.
\end{theorem}

\begin{remark}\label{rm:1.2}
The classification of Legendrian submanifolds with $C$-parallel second fundamental
form in the Sasakian space forms has been fulfilled. For the details see Theorem 4.1
in \cite{HY}.
\end{remark}

\vskip1mm
From Theorem \ref{thm:1.3}, we get a new and global geometric characterization of the
contact Whitney spheres in $\tilde{N}^{2n+1}(\tilde{c})$:

\begin{corollary}\label{cor:1.2}
Let $x:M^n\rightarrow\tilde{N}^{2n+1}(\tilde{c})\ (n\ge2)$ be an $n$-dimensional compact
Legendrian submanifold with non-$C$-parallel mean curvature vector field. Then,
\begin{equation}\label{eqn:1.11}
\int_{M^n}{\rm Ric}\,(\varphi\vec{H},\varphi\vec{H})~dV_{M^n}=
\tfrac{(n-1)(n+2)}{3n^2}\int_{M^n}\|\bar{\nabla}^\xi h\|^2~dV_{M^n}
\end{equation}
holds if and only if $x(M^n)$ is a contact Whitney sphere in $\tilde{N}^{2n+1}(\tilde{c})$ .
\end{corollary}

%%%%%%%%%%%%%%%%%%%%%%%%%%%%%%%%%%%%%%%%%%%%%%%%%%%%%%%%%%%%%%
\section{Preliminaries}\label{sect:2}
%%%%%%%%%%%%%%%%%%%%%%%%%%%%%%%%%%%%%%%%%%%%%%%%%%%%%%%%%%%%%%

In this section, we first briefly review some of the basic notions about Lagrangian
submanifolds in the complex space form $N^n(4c)$ and Legendrian submanifolds in
the Sasakian space form $\tilde{N}^{2n+1}(\tilde{c})$, respectively. Then, we state
a classical formula due to K. Yano that we need in the proof of our theorems.

Let $M^n\hookrightarrow N^n(4c)$ (resp. $M^n\hookrightarrow \tilde{N}^{2n+1}(\tilde{c})$)
be an isometric immersion from an $n$-dimensional Riemannian manifold $M^n$ into the
$n$-dimensional complex space form $N^n(4c)$ of constant holomorphic sectional curvature
$4c$ (resp. the $(2n+1)$-dimensional Sasakian space form $\tilde{N}^{2n+1}(\tilde{c})$ of
constant $\varphi$-section curvature $\tilde{c}$). For simplicity, we denote by the same
notation $g$ the Riemannian metric on $M^n$, $N^n(4c)$ and $\tilde{N}^{2n+1}(\tilde{c})$.
Let $\nabla$ (resp. $\bar\nabla$) be the Levi-Civita connection of $M^n$ (resp. $N^n(4c)$
and $\tilde{N}^{2n+1}(\tilde{c})$). Then, for both $M^n\hookrightarrow N^n(4c)$ and
$M^n\hookrightarrow \tilde{N}^{2n+1}(\tilde{c})$, we have the Gauss and Weingarten formulas:
\begin{align}\label{eqn:2.1}
\bar\nabla_XY=\nabla_XY+h(X,Y),\ \ \bar\nabla_XV=-A_VX+\nabla_X^\perp V
\end{align}
for any tangent vector fields $X,Y\in TM^n$ and normal vector field $V\in T^\perp M^n$. Here,
$\nabla^\bot$ denotes the normal connection in the normal bundle $T^\perp M$, $h$ (resp. $A_V$)
denotes the second fundamental form (resp. the shape operator with respect to $V$) of
$M^n\hookrightarrow N^n(4c)$ (resp. $M^n\hookrightarrow \tilde{N}^{2n+1}(\tilde{c})$). From
\eqref{eqn:2.1}, we have the relation
\begin{equation}\label{eqn:2.2}
g(h(X,Y),V)=g(A_VX,Y).
\end{equation}

%%%%%%%%%%%%%%%%%%%%%%%%%%%%%%%%%%%%%%%%%%%%%%%%%%%%%%%%%%%%%%
\subsection{Lagrangian submanifolds of the complex space form $N^n(4c)$}\label{sect:2.1}~
%%%%%%%%%%%%%%%%%%%%%%%%%%%%%%%%%%%%%%%%%%%%%%%%%%%%%%%%%%%%%%

The curvature tensor $\bar{R}(X,Y)Z:=\bar{\nabla}_{X}\bar{\nabla}_{Y}Z-\bar{\nabla}_{Y}\bar{\nabla}_{X}Z-\bar{\nabla}_{[X,Y]}Z$
of $N^n(4c)$ has the following expression:
\begin{equation}\label{eqn:2.3}
\begin{aligned}
\bar{R}(X,Y)Z=c\big[&g(Y,Z)X-g(X,Z)Y\\
&+g(JY,Z)JX-g(JX,Z)JY-2g(JX,Y)JZ\big].
\end{aligned}
\end{equation}

Let $M^n\hookrightarrow N^n(4c)$ be a Lagrangian immersion. Then, we have (cf. e.g. \cite{LV})
\begin{align}\label{eqn:2.4}
\nabla_X^\perp JY=J\nabla^{\perp}_XY,\quad A_{JX}Y=-Jh(X,Y)=A_{JY}X,
\end{align}
and thus $g(h(X,Y),JZ)$ is totally symmetric in $X$, $Y$ and $Z$:
\begin{equation}\label{eqn:2.5}
g(h(X,Y),JZ)=g(h(Y,Z),JX)=g(h(Z,X),JY).
\end{equation}

We choose a local adapted Lagrangian frame field
$\{e_1,...,e_n,e_{1^*},\ldots,e_{n^*}\}$ such that $e_1,\ldots,e_n$ are
orthonormal tangent vector fields, and $e_{1^*}=Je_1,\ldots,e_{n^*}=Je_n$
are orthonormal normal vector fields of $M^n\hookrightarrow N^n(4c)$,
respectively. In follows we shall make use of the indices convention:
$i^*=n+i,\ \ 1\le i,j,k,\ldots\le n$.

Denote by $\{\omega^1,\ldots,\omega^n\}$ the dual frame of $\{e_1,\ldots,e_n\}$.
Let $\omega_i^j$ and $\omega_{i^*}^{j^*}$ denote the connection
$1$-forms of $TM^n$ and $T^\perp M^n$, respectively:
$$
\nabla e_i=\sum^n_{j=1}\omega_{i}^je_j, \ \
\nabla^\perp e_{i^*}=\sum_{j=1}^n\omega_{i^*}^{j^*}e_{j^*},\ \ 1\le i\le n,
$$
where $\omega_i^j+\omega^i_j=0$ and by \eqref{eqn:2.4} it holds that
$\omega_i^j=\omega_{i^*}^{j^*}$.
Put $h^{k^*}_{ij}=g(h(e_i,e_j),Je_k)$. From \eqref{eqn:2.5}, we see that
\begin{equation}\label{eqn:2.6}
h^{k^*}_{ij}=h^{j^*}_{ik}=h^{i^*}_{jk},\ \ 1\leq i,j,k\leq n.
\end{equation}

Let $R_{ijkl}:=g\big(R(e_i,e_j)e_l,e_k\big)$ and
$R_{ijk^*l^*}:=g\big(R^\perp(e_i,e_j)e_{l^*},e_{k^*}\big)$ be the components of
the curvature tensors of $\nabla$ and $\nabla^\bot$, respectively. Then the equations
of Gauss, Ricci and Codazzi of $M^n\hookrightarrow N^n(4c)$ are given by
\begin{equation}\label{eqn:2.7}
R_{ijkl}=c(\delta_{ik}\delta_{jl}-\delta_{il}\delta_{jk})
         +\sum_{m}(h^{m^*}_{ik}h^{m^*}_{jl}-h^{m^*}_{il}h^{m^*}_{jk}),
\end{equation}
\begin{equation}\label{eqn:2.8}
R_{ijk^*l^*}=c(\delta_{ik}\delta_{jl}-\delta_{il}\delta_{jk})
+\sum_{m=1}^n(h^{m^*}_{ik}h^{m^*}_{jl}-h^{m^*}_{il}h^{m^*}_{jk})=R_{ijkl},
\end{equation}
\begin{gather}
h^{l^*}_{ij,k}=h^{l^*}_{ik,j},\label{eqn:2.9}
\end{gather}
where $h^{l^*}_{ij,k}$ denotes the components of the covariant differentiation
of $h$, namely $\bar{\nabla}h$, defined by
\begin{equation}\label{eqn:2.10}
\sum_{l=1}^nh^{l^*}_{ij,k}e_{l^*}:=\nabla^{\perp}_{e_k}\big(h(e_i,e_j)\big)
-h(\nabla_{e_k}e_i,e_j)-h(e_i,\nabla_{e_k}e_j).
\end{equation}

The mean curvature vector field $\vec{H}$ of $M^n\hookrightarrow N^n(4c)$
is defined by
\begin{equation}\label{eqn:2.11}
\vec{H}:=\tfrac1n\sum_{i=1}^nh(e_i,e_i)=:\sum_{j=1}^nH^{j^*}e_{j^*},\ \
H^{j^*}=\tfrac1n\sum_{i=1}^nh^{j^*}_{ii},\ \ 1\le j\le n.
\end{equation}

Put $\nabla^\perp_{e_i}\vec{H}=\sum\limits_{j=1}^nH^{j^*}_{,i}e_{j^*}$, $1\le i\le n$.
From \eqref{eqn:2.6} and \eqref{eqn:2.9}, we obtain
\begin{equation}\label{eqn:2.12}
H^{j^*}_{,i}=H^{i^*}_{,j},\ \ 1\le i,j\le n.
\end{equation}

%%%%%%%%%%%%%%%%%%%%%%%%%%%%%%%%%%%%%%%%%%%%%%%%%%%%%%%%%%%%%%
\subsection{Legendrian submanifolds of the Sasakian space form $\tilde{N}^{2n+1}(\tilde{c})$}\label{sect:2.2}~
%%%%%%%%%%%%%%%%%%%%%%%%%%%%%%%%%%%%%%%%%%%%%%%%%%%%%%%%%%%%%%

The following facts of this subsection are referred to e.g. \cite{HY}. The curvature
tensor of the Sasakian space form $\tilde{N}^{2n+1}(\tilde{c})$ is given by
\begin{equation}\label{eqn:2.13}
\begin{split}
\bar{R}(X,Y)Z=&\tfrac{\tilde{c}+3}{4}[g(Y,Z)X-g(X,Z)Y]+\tfrac{\tilde{c}-1}{4}\big[\eta(X)\eta(Z)Y\\
&-\eta(Y)\eta(Z)X+g(X,Z)\eta(Y)\xi-g(Y,Z)\eta(X)\xi\\
&+g(\varphi Y,Z)\varphi X-g(\varphi X,Z)\varphi Y
 +2g(X,\varphi Y)\varphi Z\big].
\end{split}
\end{equation}

Moreover, for tangent vector fields $X,Y$ of $\tilde{N}^{2n+1}(\tilde{c})$, the Sasakian structure
$(\varphi,\xi,\eta,g)$ of $\tilde{N}^{2n+1}(\tilde{c})$ satisfy:
\begin{equation}\label{eqn:2.14}
\left\{
\begin{aligned}
&\eta(X)=g(X,\xi),\ \ \varphi\xi=0,\ \ \eta(\varphi X)=0,\\
&\varphi^2X=-X+\eta(X)\xi,\ \ d\eta(X,Y)=g(X,\varphi Y),\\
&g(\varphi X,\varphi Y)=g(X,Y)-\eta(X)\eta(Y),\ \ {\rm rank}\,(\varphi)=2n,\\
&\bar{\nabla}_{X}\xi=-\varphi X,\ \ (\bar{\nabla}_{X}\varphi)Y =g(X,Y)\xi-\eta(Y)X.
\end{aligned}
\right.
\end{equation}

Let $M^n\hookrightarrow \tilde{N}^{2n+1}(\tilde{c})$ be a Legendrian immersion.
Then, we have
\begin{equation}\label{eqn:2.15}
A_{\varphi Y}X=-\varphi h(X,Y),\quad
\nabla^{\bot}_{X}\varphi Y=\varphi\nabla_XY+g(X,Y)\xi.
\end{equation}

In follows we shall make the following convention on range of indices:
$$
\alpha^*=\alpha+n;\ \ 1\le i,j,k,l,m\le n; \ \ 1\le \alpha,\beta\le n+1.
$$

We choose a local {\it Legendre frame field} $\{e_1,\ldots,e_n,e_{1^*},
\ldots,e_{n^*},e_{2n+1}=\xi\}$ along $M^n\hookrightarrow \tilde{N}^{2n+1}(\tilde{c})$
such that $\{e_i\}_{i=1}^n$ is an orthonormal frame field of $M^n$, and $\{e_{1^*}=\varphi
e_1, \ldots,e_{n^*}=\varphi e_n,e_{2n+1}=\xi\}$ are the orthonormal normal vector fields
of $M^n\hookrightarrow\tilde{N}^{2n+1}(\tilde{c})$. Let $\omega_i^j$ and $\omega_{\alpha^*}^{\beta^*}$
denote the connection $1$-forms of $TM^n$ and $T^\perp M^n$, respectively:
$$
\nabla e_i=\sum^n_{j=1}\omega_{i}^je_j, \ \
\nabla^\perp e_{\alpha^*}=\sum_{\beta=1}^{n+1}\omega_{\alpha^*}^{\beta^*}e_{\beta^*},\ \ 1\le i\le n,\ \ 1\le\alpha\le n+1,
$$
where $\omega_i^j+\omega_j^i=0$ and $\omega_{\alpha^*}^{\beta^*}
+\omega_{\beta^*}^{\alpha^*}=0$. Moreover, by \eqref{eqn:2.15}, we have
$\omega_i^j=\omega_{i^*}^{j^*}$ and $\omega_{i^*}^{2n+1}=\omega^i$.
Put $h^{k^*}_{ij}=g(h(e_i,e_j),\varphi e_k)$ and $h^{2n+1}_{ij}=g(h(e_i,e_j),e_{2n+1})$.
From \eqref{eqn:2.14} and \eqref{eqn:2.15}, we have
\begin{equation}\label{eqn:2.16}
h^{k^*}_{ij}=h^{j^*}_{ik}=h^{i^*}_{jk},\quad h^{2n+1}_{ij}=0,\ \ 1\le i,j,k\le n.
\end{equation}

Now, with the same notations as in the preceding subsection, the equations of Gauss, Ricci
and Codazzi of $M^n\hookrightarrow \tilde{N}^{2n+1}(\tilde{c})$ are as follows:
\begin{equation}\label{eqn:2.17}
R_{ijkl}=\tfrac{\tilde{c}+3}4(\delta_{ik}\delta_{jl}-\delta_{il}\delta_{jk})
+\sum_{m=1}^n(h_{ik}^{m^*}h_{jl}^{m^*}-h_{il}^{m^*}h_{jk}^{m^*}),
\end{equation}
\begin{equation}\label{eqn:2.18}
R_{ijk^*l^*}=\tfrac{\tilde{c}-1}4(\delta_{ik}\delta_{jl}-\delta_{il}\delta_{jk})
+\sum_{m=1}^n(h^{m^*}_{ik}h^{m^*}_{jl}-h^{m^*}_{il}h^{m^*}_{jk}), \ \ R_{ijk^*(2n+1)}=0,
\end{equation}
\begin{equation}\label{eqn:2.19}
h^{\alpha^*}_{ij,k}=h^{\alpha^*}_{ik,j},
\end{equation}
where as usual $h^{\alpha^*}_{ij,k}$ is defined by
\begin{equation}\label{eqn:2.20}
\sum_{\alpha=1}^{n+1}h^{\alpha^*}_{ij,k}e_{\alpha^*}:=\nabla_{e_k}^\perp\big(h(e_i,e_j)\big)
-h(\nabla_{e_k}e_i,e_j)-h(e_i,\nabla_{e_k}e_j),\ \ 1\le i,j,k\le n.
\end{equation}
Moreover, associated to $\nabla,\,\nabla^\perp$ and $\bar\nabla$, we can naturally define
a modified covariant differentiation $\bar\nabla^\xi h$ of the second fundamental form
by
\begin{equation}\label{eqn:2.21}
(\bar\nabla^{\xi}_Xh)(Y,Z):=\nabla_X^\perp(h(Y,Z))-h(\nabla_XY,Z)
-h(Y,\nabla_XZ)-g(h(Y,Z),\varphi X)\xi.
\end{equation}

Recall that the second fundamental form $h$ of $M^n\hookrightarrow \tilde{N}^{2n+1}(\tilde{c})$
is said to be $C$-parallel if and only if $\bar\nabla^\xi h=0$ (cf. \cite{HY}). Actually, we have
$g((\bar\nabla^{\xi}_Xh)(Y,Z),\xi)=0$ for any $X,Y,Z\in TM^n$. Thus, we
can denote
\begin{equation}\label{eqn:2.22}
(\bar\nabla^{\xi}_{e_k}h)(e_i,e_j):=\sum_{l=1}^n\bar{h}^{l^*}_{ij,k}e_{l^*},\ \ 1\le i,j,k\le n.
\end{equation}

Then, by \eqref{eqn:2.20}, \eqref{eqn:2.21} and the above discussions, we have
\begin{equation}\label{eqn:2.23}
h_{ij,k}^{(n+1)^*}=h_{ij}^{k^*},\ \ h^{l^*}_{ij,k}=\bar{h}^{l^*}_{ij,k},\ \ \forall\, i,j,k,l.
\end{equation}

From \eqref{eqn:2.16}, the mean curvature vector $\vec{H}$ of
$M^n\hookrightarrow\tilde{N}^{2n+1}(\tilde{c})$ becomes:
\begin{equation}\label{eqn:2.24}
\vec{H}=\tfrac1n\sum_{i=1}^nh(e_i,e_i)=\sum_{k=1}^nH^{k^*}e_{k^*},\
\ H^{k^*}:=\tfrac1n\sum_{i=1}^nh^{k^*}_{ii},\ \ 1\le k\le n.
\end{equation}

Put
$$
\nabla^\perp_{e_i} \vec{H}=\sum_{\alpha=1}^{n+1} H^{\alpha^*}_{,i}e_{\alpha^*},\ \
\bar\nabla^\xi_{e_i}\vec{H}:=\nabla^\perp_{e_i}\vec{H}-g(\vec{H},e_{i^*})\xi
=:\sum_{k=1}^n\bar{H}^{k^*}_{,i}e_{k^*},\ \ 1\le i\le n.
$$

From \eqref{eqn:2.16}, \eqref{eqn:2.19} and \eqref{eqn:2.23}, we get
\begin{equation}\label{eqn:2.25}
H^{j^*}_{,i}=H^{i^*}_{,j},\quad \bar{H}^{j^*}_{,i}=H^{j^*}_{,i},\ \ 1\le i,j\le n.
\end{equation}

%%%%%%%%%%%%%%%%%%%%%%%%%%%%%%%%%%%%%%%%%%%%%%%%%%%%%%%%%%%%%%
\subsection{Yano's formula}\label{sect:2.3}~
%%%%%%%%%%%%%%%%%%%%%%%%%%%%%%%%%%%%%%%%%%%%%%%%%%%%%%%%%%%%%%

In order to prove Theorem \ref{thm:1.2} and Theorem \ref{thm:1.3}, we still need
the following useful formula due to K. Yano \cite{Y}. A simply proof is referred
also to \cite{HMVY}.

\begin{lemma}[cf. Lemma 5.1 of \cite{HMVY}]\label{lem:2.1}
Let $(M,g)$ be a Riemannian manifold with Levi-Civita connection $\nabla$. Then,
for any tangent vector field $X$ on $M$, it holds that
\begin{equation}\label{eqn:2.26}
\begin{aligned}
{\rm div}(\nabla_XX-({\rm div}X)X)={\rm Ric}\,(X,X)+\tfrac{1}{2}\|\mathcal{L}_Xg\|^2
-\|\nabla X\|^2-({\rm div}X)^2,
\end{aligned}
\end{equation}
where $\mathcal{L}_Xg$ is the Lie derivative of $g$ with respect to $X$ and $\|\cdot\|$
denotes the length with respect to $g$.
\end{lemma}

%%%%%%%%%%%%%%%%%%%%%%%%%%%%%%%%%%%%%%%%%%%%%%%%%%%%%%%%%%%%%%
\section{Proof of Theorem \ref{thm:1.2}}\label{sect:3}
%%%%%%%%%%%%%%%%%%%%%%%%%%%%%%%%%%%%%%%%%%%%%%%%%%%%%%%%%%%%%%

First of all, we state the following simple fact without proof.
\begin{lemma}\label{lem:3.1}
Let $x:M^n\rightarrow N^n(4c)$ be an $n$-dimensional Lagrangian submanifold
with mean curvature vector field $\vec{H}$. Then, it holds that
\begin{equation}\label{eqn:3.1}
\begin{aligned}
\|\nabla J\vec{H}\|^2\ge\tfrac1n({\rm div}J\vec{H})^2.
\end{aligned}
\end{equation}
Moreover, the equality in \eqref{eqn:3.1} holds if and only if
$\nabla J\vec{H}=\frac1n({\rm div}\,J\vec{H})\,{\rm id}$, i.e.,
$J\vec{H}$ is a conformal vector field on $M^n$, or equivalently,
$M^n$ is a Lagrangian submanifold with conformal Maslov form.
\end{lemma}

We also need the following result due to Li-Vrancken \cite{LV}:
\begin{lemma}[cf. Lemma 3.2 in \cite{LV}]\label{lem:3.2}
Let $x:M^n\rightarrow N^n(4c)$ be an $n$-dimensional Lagrangian submanifold
with mean curvature tensor $\vec{H}$. Then, it holds that
\begin{equation}\label{eqn:3.2}
\|\bar\nabla h\|^2\ge\tfrac{3n^2}{n+2}\|\nabla^\perp\vec{H}\|^2.
\end{equation}

Moreover, the equality in \eqref{eqn:3.2} holds if and only if
\begin{equation}\label{eqn:3.3}
(\bar\nabla_Z h)(X,Y)=\tfrac{n}{n+2}\big[g(Y,Z)\nabla_X^\perp\vec{H}+g(X,Z)\nabla_Y^\perp\vec{H}+g(X,Y)\nabla_Z^\perp\vec{H}\big].
\end{equation}

\end{lemma}

\vskip1mm
Now, we are ready to complete the proof of Theorem \ref{thm:1.2}.

\begin{proof}[Proof of Theorem \ref{thm:1.2}]
Let $M^n\hookrightarrow N^n(4c)$ be a compact Lagrangian submanifold
and $\{e_1,...,e_n,e_{1^*},\ldots,e_{n^*}\}$ be a local adapted Lagrangian
frame field along $M^n$. From \eqref{eqn:2.4} and that
$\nabla^\perp_{e_i}\vec{H}=\sum\limits_{j=1}^nH^{j^*}_{,i}e_{j^*}$, we have
\begin{equation}\label{eqn:3.4}
\begin{aligned}
\|\nabla J\vec{H}\|^2=\|\nabla^{\perp}\vec{H}\|^2=\sum_{i,j=1}^n(H^{j^*}_{,i})^2.
\end{aligned}
\end{equation}

Then, by \eqref{eqn:2.12}, calculating the squared length of the Lie derivative
$\mathcal{L}_{J\vec{H}}g$ of $g$ with respect to $J\vec{H}$, we obtain
\begin{equation}\label{eqn:3.5}
\begin{aligned}
\|\mathcal{L}_{J\vec{H}}g\|^2&=\sum_{i,j=1}^n\big[(\mathcal{L}_{J\vec{H}}g)(e_i,e_j)\big]^2
=\sum_{i,j=1}^n\big(H^{j^*}_{,i}+H^{i^*}_{,j}\big)^2=4\|\nabla^{\perp}\vec{H}\|^2.
\end{aligned}
\end{equation}

Thus, we can apply Lemma \ref{lem:2.1} and \eqref{eqn:3.1} to obtain that
\begin{equation}\label{eqn:3.6}
\begin{aligned}
{\rm div}(\nabla_{J\vec{H}}J\vec{H}-({\rm div}J\vec{H})J\vec{H})
&={\rm Ric}\,(J\vec{H},J\vec{H})+\|\nabla J\vec{H}\|^2-({\rm div}J\vec{H})^2\\
&\ge{\rm Ric}\,(J\vec{H},J\vec{H})-(n-1)\|\nabla^{\perp}\vec{H}\|^2,
\end{aligned}
\end{equation}
where the equality in \eqref{eqn:3.6} holds if and only if
$\nabla J\vec{H}=\frac1n({\rm div}\,J\vec{H})\,{\rm id}$,
or equivalently, $M^n$ is a Lagrangian submanifold with conformal Maslov form.

From \eqref{eqn:3.6}, by further applying Lemma \ref{lem:3.2}, we get
\begin{equation}\label{eqn:3.7}
{\rm div}(\nabla_{J\vec{H}}J\vec{H}-({\rm div}J\vec{H})J\vec{H})
\ge{\rm Ric}\,(J\vec{H},J\vec{H})-\tfrac{(n-1)(n+2)}{3n^2}\|\bar\nabla h\|^2,
\end{equation}
where the equality holds if and only if both $\nabla J\vec{H}=\frac1n({\rm div}\,J\vec{H})\,{\rm id}$
and \eqref{eqn:3.3} hold.

By the compactness of $M^n$, we can integrate the inequality \eqref{eqn:3.7} over $M^n$.
Then, applying for the divergence theorem, we obtain the integral inequality \eqref{eqn:1.4}.

It is easily seen that the equality holds in \eqref{eqn:1.4} if and only if
the equality in \eqref{eqn:3.2} holds identically. Thus, according to Main Theorem in \cite{LV},
equality in \eqref{eqn:1.4} holds if and only if either $x(M^n)$
is of parallel second fundamental form, or $x(M^n)$ is one of the Whitney spheres
in $N^n(4c)$.

This completes the proof of Theorem \ref{thm:1.2}.
\end{proof}

%%%%%%%%%%%%%%%%%%%%%%%%%%%%%%%%%%%%%%%%%%%%%%%%%%%%%%%%%%%%%%
\section{Proof of Theorem \ref{thm:1.3}}\label{sect:4}
%%%%%%%%%%%%%%%%%%%%%%%%%%%%%%%%%%%%%%%%%%%%%%%%%%%%%%%%%%%%%%

Let $x:M^n\rightarrow\tilde{N}^{2n+1}(\tilde{c})$ be an $n$-dimensional
Legendrian submanifold in the Sasakian space form $\tilde{N}^{2n+1}(\tilde{c})$
with Sasakian structure $(\varphi,\xi,\eta,g)$. First of all, similar to
Lemma \ref{lem:3.1}, we have the following simple result.
\begin{lemma}\label{lem:4.1}
Let $x:M^n\rightarrow\tilde{N}^{2n+1}(\tilde{c})$ be an $n$-dimensional
Legendian submanifold with mean curvature vector field $\vec{H}$. Then,
it holds that
\begin{equation}\label{eqn:4.1}
\begin{aligned}
\|\nabla(\varphi\vec{H})\|^2\ge\tfrac1n({\rm div}\,\varphi\vec{H})^2.
\end{aligned}
\end{equation}
Moreover, the equality in \eqref{eqn:4.1} holds if and only if
$\nabla(\varphi\vec{H})=\frac1n({\rm div}\,\varphi\vec{H})\,{\rm id}$, i.e.,
$\varphi\vec{H}$ is a conformal vector field on $M^n$.
\end{lemma}

We also need the following result:

\begin{lemma}[cf. Lemma 3.3 in \cite{HY}]\label{lem:4.2}
Let $x:M^n\rightarrow\tilde{N}^{2n+1}(\tilde{c})$ be an $n$-dimensional Legendrian
submanifold with second fundamental form $h$ and mean curvature
vector field $\vec{H}$. Then, it holds that
\begin{equation}\label{eqn:4.2}
\|\bar\nabla^{\xi} h\|^2\geq \tfrac{3n^2}{n+2}\|\bar\nabla^{\xi}\vec{H}\|^2,
\end{equation}
where, with respect to a local Legendre frame field $\{e_A\}_{A=1}^{2n+1}$,
$$
\|\bar\nabla^{\xi} h \|^2=\sum_{i,j,k,l=1}^n(h^{l^*}_{ij,k})^2,\ \
\|\bar\nabla^{\xi}\vec{H}\|^2=\sum_{i,j=1}^n(H^{j^*}_{,i})^2.
$$

Moreover, the equality in \eqref{eqn:4.2} holds if and only if
\begin{equation}\label{eqn:4.3}
h^{l^*}_{ij,k}=\tfrac{n}{n+2}\big(H^{l^*}_{,i}\delta_{jk}+H^{l^*}_{,j}\delta_{ik}
+H^{l^*}_{,k}\delta_{ij}\big),\ \ 1\le i,j,k,l\le n.
\end{equation}
\end{lemma}

\vskip2mm
Now, we are ready to complete the proof of Theorem \ref{thm:1.3}.

\begin{proof}[Proof of Theorem \ref{thm:1.3}]

Let $x:M^n\rightarrow\tilde{N}^{2n+1}(\tilde{c})$ be a compact $n$-dimensional
Legendrian submanifold and $\{e_1,\ldots,e_n,e_{1^*},\ldots,e_{n^*},e_{2n+1}=\xi\}$
be a local adapted Legendre frame field along $M^n$. By definition, we have
\begin{equation}\label{eqn:4.4}
\begin{aligned}
\|\nabla(\varphi\vec{H})\|^2=\sum_{i,j=1}^n(g(\nabla_{e_i}(\varphi\vec{H}),e_j))^2
=\sum_{i,j=1}^n(\bar{H}^{j^*}_{,i})^2=\|\bar\nabla^\xi\vec{H}\|^2.
\end{aligned}
\end{equation}

Then, by \eqref{eqn:2.25}, calculating the squared length of the Lie derivative
$\mathcal{L}_{\varphi\vec{H}}g$ of $g$ with respect to $\varphi\vec{H}$, we obtain
\begin{equation}\label{eqn:4.5}
\begin{aligned}
\|\mathcal{L}_{\varphi\vec{H}}g\|^2
=\sum_{i,j=1}^n\big[(\mathcal{L}_{\varphi\vec{H}}g)(e_i,e_j)\big]^2
=\sum_{i,j=1}^n\big(H^{j^*}_{,i}+H^{i^*}_{,j}\big)^2=4\|\nabla(\varphi\vec{H})\|^2.
\end{aligned}
\end{equation}

Thus, we can apply Lemma \ref{lem:2.1} and \eqref{eqn:4.1} to obtain that
\begin{equation}\label{eqn:4.6}
\begin{aligned}
{\rm div}(\nabla_{\varphi\vec{H}}(\varphi\vec{H})-({\rm div}\,\varphi\vec{H})\varphi\vec{H})
&={\rm Ric}\,(\varphi\vec{H},\varphi\vec{H})+\|\nabla (\varphi\vec{H})\|^2-({\rm div}\varphi\vec{H})^2\\
&\ge{\rm Ric}\,(\varphi\vec{H},\varphi\vec{H})-(n-1)\|\nabla(\varphi\vec{H})\|^2,
\end{aligned}
\end{equation}
where the equality in \eqref{eqn:4.6} holds if and only if
$\nabla(\varphi\vec{H})=\frac1n({\rm div}\,\varphi\vec{H})\,{\rm id}$.

From \eqref{eqn:4.6} and that $\|\nabla(\varphi\vec{H})\|^2=\|\bar\nabla^\xi\vec{H}\|^2$,
by further applying Lemma \ref{lem:4.2}, we get
\begin{equation}\label{eqn:4.7}
{\rm div}(\nabla_{\varphi\vec{H}}(\varphi\vec{H})-({\rm div}\,\varphi\vec{H})\varphi\vec{H})
\ge{\rm Ric}\,(\varphi\vec{H},\varphi\vec{H})-\tfrac{(n-1)(n+2)}{3n^2}\|\bar\nabla^\xi h\|^2,
\end{equation}
where the equality holds if and only if both $\nabla(\varphi\vec{H})=\frac1n({\rm div}\,\varphi\vec{H})\,{\rm id}$
and \eqref{eqn:4.3} hold.

By the compactness of $M^n$, we can integrate the inequality \eqref{eqn:4.7}
over $M^n$. Then, applying for the divergence theorem, we obtain the integral
inequality \eqref{eqn:1.10}.

It is easily seen from the above arguments that the equality in \eqref{eqn:1.10} holds
if and only if the equality in \eqref{eqn:4.2} holds identically. Thus, according to
Theorem 1.1 in \cite{HY}, equality in \eqref{eqn:1.10} holds if and only if either
$x(M^n)$ is of $C$-parallel second fundamental form, or $x(M^n)$ is one of the contact
Whitney spheres in $\tilde{N}^{2n+1}(\tilde{c})$.

This completes the proof of Theorem \ref{thm:1.3}.
\end{proof}

\vskip2mm
As final remarks, we would mention that all the Whitney spheres in the complex space
forms are conformally equivalent to the round sphere (cf. \cite{RU} and \cite{CMU}).
Now, for any one of the contact Whitney spheres, $x:\mathbb{S}^n\rightarrow\tilde{N}^{2n+1}(\tilde{c})$,
its second fundamental form $h$ has the expression \eqref{eqn:1.9}. Thus, by using the
Gauss equation and direct calculations, we can immediately obtain the following

\begin{theorem}
The sectional curvatures of the contact Whitney spheres are not constant. Nevertheless, all
the contact Whitney spheres in each of the Sasakian space forms are conformally equivalent
to the round sphere.
\end{theorem}

%%%%%%%%%%%%%%%%%%%%%%%%%%%%%%%%%%%%%%%%%%%%%%%%%%%%%%%%%%%%%%

\end{document}